\documentclass[a4paper,12pt]{amsart}
\usepackage[utf8]{inputenc}
\usepackage{amssymb, mathrsfs, amsfonts, amsmath}
\usepackage{mathtools} 
\usepackage{bbm}
\usepackage{hyperref}

\newtheorem{theorem}{Theorem}
\newtheorem{conjecture}{Conjecture}

\newtheorem{lemm}{Lemma}

\newtheorem{prop}{Proposition}

\newtheorem*{problem1}{Fourier Eigenvalue Linear Programming Problem}

\newcommand{\Co}{\mathbb{C}}
\newcommand{\R}{\mathbb{R}}
\newcommand{\N}{\mathbb{N}}
\newcommand{\Z}{\mathbb{Z}}
\newcommand{\A}{\mathcal{A}}
\newcommand{\mmd}{\mathrm{d}}

\newcommand{\Sch}{\mathcal{S}}

\renewcommand{\d}{\,{\rm d}} 

\numberwithin{equation}{section}

\title[regularity and mass concentration]{On regularity and mass concentration phenomena for the sign uncertainty principle}

\author[F. Gon\c calves]{Felipe Gon\c calves}
\author[D. Oliveira e Silva]{Diogo Oliveira e Silva}
\author[J. P. G. Ramos]{Jo\~ao P. G. Ramos}

\address{
	Hausdorff Center for Mathematics\\
        53115 Bonn, Germany}
\email{goncalve@math.uni-bonn.de}

\address{
        School of Mathematics\\
	University of Birmingham\\
	B15 2TT, England, UK}
\email{d.oliveiraesilva@bham.ac.uk}

\address{ 
Instituto Nacional de Matem\'atica Pura e Aplicada \\
22460-320 Rio de Janeiro, Brazil}
\email{jpgramos@impa.br}

\begin{document} 

\begin{abstract} 
The sign uncertainty principle of Bourgain, Clozel \& Kahane asserts that if a function $f:\mathbb{R}^d\to \mathbb{R}$ and its Fourier transform $\widehat{f}$ are nonpositive at the origin and not identically zero, then they cannot both be nonnegative outside an arbitrarily small neighborhood of the origin.
In this article, we establish some equivalent formulations of the sign uncertainty principle, and in particular prove that minimizing sequences exist within the Schwartz class when $d=1$.
We further address a complementary sign uncertainty principle, and show that corresponding near-minimizers concentrate a universal proportion of their positive mass near the origin in all dimensions.
\end{abstract}

\maketitle

\section{Introduction} 

Motivated by a problem in the theory of zeta functions over algebraic number fields, Bourgain, Clozel \& Kahane \cite{BCK10} investigated the class of functions $\mathcal A_+(d)$, defined as follows.
Given $d\geq 1$, a function $f:\R^d\to\R$ is said to be {\it eventually nonnegative} if $f(x)\geq 0$ for all sufficiently large $|x|$.
Normalize the  Fourier transform,
\begin{equation}\label{eq:FTnorm}
\widehat{f}(\xi)=\int_{\R^d} f(x)e^{-2\pi i \langle x,\xi\rangle}\, \d x,
\end{equation}
where $\langle\cdot,\cdot\rangle$ represents the usual inner product in $\R^d$.
Let $\mathcal A_+(d)$ denote the set of functions $f:\R^d\to\R$ which satisfy the following conditions: 
\begin{itemize}
\item $f\in L^1(\R^d)$, $\widehat{f}\in L^1(\R^d)$, and $\widehat{f}$ is real-valued (i.e.\@ $f$ is even);
\item $f$ is eventually nonnegative while $\widehat{f}(0)\leq 0$;
\item $\widehat f$ is eventually nonnegative while ${f}(0)\leq 0$.
 \end{itemize}
 Note that any function $f\in\mathcal A_+(d)$ is uniformly continuous.
Consider the quantity
\[r(f):=\inf\{r> 0: f(x)\geq 0 \text{ if } |x|\geq r\},\] 
which corresponds to the radius of the last sign change of $f$.
The product $r(f)r(\widehat f)$ is unchanged if we replace $f$ with $x\mapsto f(\lambda x)$ for some $\lambda>0$, and thus becomes a natural object to consider. 
One of the initial observations in \cite{BCK10} is that the quantity
\begin{equation}\label{eq:+1-up}
 \mathbb{A}_+(d) := \inf_{f \in \mathcal{A}_+(d)\setminus\{\bf 0\}} \sqrt{r(f) r(\widehat{f})}
\end{equation}
is uniformly bounded from below away from zero. 
In fact,  the following two-sided inequality is established in \cite[\S 3]{BCK10}: 
\begin{equation}\label{eq:lim-inf-sup-plus} 
 \frac{1}{\sqrt{2\pi e}}
\leq  \liminf_{d \to \infty} \frac{\mathbb{A}_+(d)}{\sqrt{d}}
\leq \limsup_{d \to \infty} \frac{ \mathbb{A}_+(d)}{\sqrt{d}} 
\leq \frac{1}{\sqrt{2 \pi}}. 
\end{equation}
In particular, the radii $r(f),r(\widehat f)$ of the last sign change of $f,\widehat f$, respectively, cannot both be made arbitrarily small, unless $f\in \mathcal A_+(d)$ is identically zero.
Consequently, the aforementioned results can be regarded as manifestations of  a {\it sign uncertainty principle}.

The sign uncertainty principle of Bourgain, Clozel \& Kahane inspired a number of subsequent works \cite{CG19, GOeSR19, GOeSS17}; see also \cite{CMS19, GIT19}.
 Gon\c calves,  Oliveira e Silva \& Steinerberger \cite{GOeSS17} proved that radial minimizers for \eqref{eq:+1-up} exist. 
 More precisely, in each dimension $d\geq 1$, they showed that there exists a radial function $f\in\mathcal A_+(d)$, satisfying $\widehat f=f, f(0)=0$, and $r(f)=\mathbb{A}_+(d)$,
and that such a minimizer must necessarily vanish at infinitely many radii greater than $\mathbb{A}_+(d)$. 
The precise shape of minimizers remained a mystery in all dimensions $d\geq 1$, until Cohn \& Gon\c calves \cite{CG19} exhibited an explicit minimizer in twelve dimensions. 
In particular, they relied on the following ingredients in order to show that $\mathbb{A}_+(12)= \sqrt{2}$: 
\begin{itemize}
 \item A Poisson-type summation formula for radial Schwartz functions $f:\R^{12}\to\Co$ based on the modular form $E_6$;
 \item An explicit construction via a remarkable integral transform discovered by Viazosvska \cite{Vi17} which turns modular forms into radial eigenfunctions of the Fourier transform. 
\end{itemize}
The first ingredient leads to the lower bound $\mathbb{A}_+(12)\ge \sqrt{2}$.
The second ingredient produces the explicit minimizer, and in particular leads to the upper bound $\mathbb{A}_+(12)\leq \sqrt{2}$. 
Moreover, the minimizer is shown to belong to the Schwartz class $\mathcal S(\R^{12})$, and
to be a $+1$ eigenfunction of the Fourier transform; see \cite[Fig.\@ 1]{CG19} for the corresponding plots. 
It then becomes natural to consider a complementary problem, associated to $-1$ eigenfunctions of the Fourier 
transform, which we now describe. 

Let $\mathcal{A}_-(d)$ denote the set of integrable functions $f: \R^d \to \R$, with integrable, real-valued Fourier transform $\widehat{f}$, such that $\widehat{f}(0) \le 0$ while $f$ is eventually nonnegative, and $f(0) \ge 0$ while $-\widehat{f}$ is eventually nonnegative.
Define the quantity 
\begin{equation}\label{eq:minA-}
 \mathbb{A}_-(d) := \inf_{f \in \mathcal{A}_-(d)\setminus\{{\bf 0}\}} \sqrt{r(f) r(\widehat{f})}.
 \end{equation}
Then \cite[Theorem 1.4]{CG19} guarantees that the chain of inequalities \eqref{eq:lim-inf-sup-plus} still holds if $\mathbb{A}_+(d)$ is replaced by $\mathbb{A}_-(d)$.
It further ensures the existence of radial minimizers for \eqref{eq:minA-} which are  $-1$ eigenfunctions of the Fourier transform, and necessarily have infinitely many zeros after the last sign change. 
Logan \cite{L2} has solved the optimization problem \eqref{eq:minA-} in the one-dimensional case $d=1$.
 Problem \eqref{eq:minA-} turns out to be closely related to the sphere packing problem and, in light of the recent breakthroughs \cite{CKMRV17, Vi17}, it has also been solved in dimensions $d \in \{8,24\}$; see \cite[Prop.\@ 7.1]{CE03} and \cite[\S 1]{CG19}. 
In particular, $\mathbb{A}_-(1)=1$, $\mathbb{A}_-(8)=\sqrt{2}$, and $\mathbb{A}_-(24)=2$.
Moreover, if $d \in \{8,24\}$, then minimizers for \eqref{eq:minA-} belong to the Schwartz class, and modulo symmetries are unique within this class.
However, no such refined regularity properties can be asserted {\it a priori} in any other dimension.

\subsection{Main results}
In this article, we investigate regularity properties and mass concentration phenomena exhibited by minimizing sequences of \eqref{eq:+1-up} and \eqref{eq:minA-}.

We use the letter $s$ to denote a sign from $\{+,-\}$, and shall sometimes identify the signs $\{+,-\}$ with the integers $\{+1,-1\}$. 
A function $f:\R^d\to\R$ is said to be {\it bandlimited} if the support of its distributional Fourier transform $\widehat f$ is compact, denoted $\textup{supp}(\widehat f) \Subset \R^d$.  
For $s\in\{+,-\}$, define the quantities
\begin{align*}
\mathbb{A}_s^{\mathcal B}(d) := \inf_{f\in\mathcal{A}_s(d)\setminus\{{\bf 0}\}\atop \textup{supp}(\widehat f) \Subset \R^d} \sqrt{r(f) r(\widehat{f})};\;\;\;
\mathbb{A}_s^{\mathcal S}(d) := \inf_{f\in \mathcal{A}_s(d)\cap \mathcal{S}(\R^d)\setminus\{{\bf 0}\} } \sqrt{r(f) r(\widehat{f})},
\end{align*}
where the infima are taken over nonzero functions in $\mathcal A_s(d)$ which are bandlimited and belong to the Schwartz space, $\mathcal S(\R^d)$, respectively. It is then natural to wonder about the relationship between the sharp constants $\mathbb{A}_s(d)$, $\mathbb{A}_s^{\mathcal B}(d)$, and $\mathbb{A}_s^{\mathcal S}(d)$.

The identities $\mathbb{A}_-(8)=\mathbb{A}_-^{\mathcal{S}}(8)$, $\mathbb{A}_-(24)=\mathbb{A}_-^{\mathcal{S}}(24)$, and $\mathbb{A}_+(12)=\mathbb{A}_+^{\mathcal{S}}(12)$ are known, simply because the corresponding minimizers in $\mathcal S(\R^d)$ have been explicitly constructed. If $(s,d)=(-,1)$, then uniqueness of minimizers (modulo symmetries) in $\mathcal S(\R)$ fails, 
but from knowledge of the corresponding minimizers one can likewise infer that $\mathbb{A}_-(1)=\mathbb{A}_-^{\mathcal{S}}(1)$. 
According to \cite[Th\'eor\`eme 3.2]{BCK10}, the two-sided inequality
\begin{equation}\label{eq:equivalence-Schwartz}
\mathbb{A}_s(d) \le \mathbb{A}_s^{\mathcal{S}}(d) \le 2\mathbb{A}_s(d)
\end{equation}
holds when $s=+1$, and the argument presented there can be easily adapted to the case $s=-1$.
As remarked in \cite{BCK10}, it is not at all clear that the first inequality in \eqref{eq:equivalence-Schwartz} should be an equality.\footnote{\cite[p.~1218]{BCK10}: ``Il n'est point \'evident que la borne $A$ d\'efinie par $A=\inf A(f)$, quand on impose de surcro\^it \`a $f$ d'appartenir \`a $\mathcal S$, co\"incide avec celle d\'efinie pour $f$ parcourant $L^1$.''} 
Our first main result settles this question for $(s,d)=(+,1)$, where all three aforementioned sign uncertainty principles are seen to be equivalent. 

\begin{theorem}\label{thm:plus-equiv} 
 $ \mathbb{A}_+(1)
 =\mathbb{A}_+^{\mathcal B}(1)
 =\mathbb{A}_+^{\mathcal{S}}(1).$
\end{theorem}

For other combinations of signs and dimensions,  barely anything is known about regularity properties of near-minimizers for $\mathbb{A}_s(d)$, and the following conjecture  remains open in its full generality.

\begin{conjecture}\label{conj1}
For any $s\in\{+,-\}$ and $d\geq 1$, it holds that $ \mathbb{A}_s(d)
 =\mathbb{A}_s^{\mathcal B}(d)
 =\mathbb{A}_s^{\mathcal{S}}(d).$
\end{conjecture}

 The proof of Theorem \ref{thm:plus-equiv} builds upon a few observations from \cite{BCK10}:
Given $f\in\mathcal A_+(d)$,
there exists $x_0\in\R^d$ satisfying $|x_0|\leq r(f)$, such that $f(x_0) <0$.
Convolving with the sum of Dirac measures
$\delta_{x_0}+\delta_{-x_0}+2\delta_{0}$ yields a new function $g \in \mathcal{A}_+(d)$, satisfying $g(0) < 0$ and $r(g) \le 2r(f)$.
This provides additional room for a further convolution with an appropriate smooth function, and ultimately enables selected minimizing sequences  to be found within a smoother function space. 
To improve on this in the one-dimensional setting, we show that any minimizer of \eqref{eq:+1-up} is strictly negative on a punctured neighborhood of the origin; in particular, the point $x_0$ can be taken arbitrarily close to the origin if $d=1$.

The previous paragraph and the explicit minimizer found in \cite{CG19} together imply the existence of minimizers which are nonpositive in a neighbourhood of the origin in dimensions $d\in\{1,12\}$. 
In fact, the minimizer in twelve dimensions is nonpositive on the open ball $B_{\sqrt 2}^{12}\subseteq\R^{12}$ centered at the origin of radius $\sqrt{2}=\mathbb{A}_+(12)$,
which makes it tempting to conjecture that such a property holds in arbitrary dimensions.
The numerical examples from \cite{CG19,GOeSR19,GOeSS17} provide further evidence for this possibility; see \cite[Fig.\@ 1]{GOeSR19} for the plot of a numerical approximation of a minimizer for $\mathbb{A}_+(1)$. 
 
An adaptation of the proof of Theorem \ref{thm:plus-equiv} yields the following improvement over \eqref{eq:equivalence-Schwartz} in higher dimensions $d>1$:
There exist constants $\delta_d\in(0,1)$, satisfying $\sqrt{d}(2\pi e)^{d/2}\delta_d\to 1$, as $d\to\infty$, such that
\begin{equation}\label{eq:improvement}
\mathbb{A}_+^{\mathcal{S}}(d) \le (2-\delta_d)\mathbb{A}_+(d). 
\end{equation}
In fact, if $d>1$, then we are able to identify {\it small}, but not {\it arbitrarily small}, values of $|x_0|$ for which a given minimizer $f$ satisfies $f(x_0)<0$.
 This is the main reason why our methods do not seem sufficient to establish Conjecture \ref{conj1} for the Schwartz class
 if $(s,d)\neq (+,1)$. For further details, see \S \ref{sec:weak-proof} below. \\

Even though our current techniques do not seem fit to establish any non-trivial equivalence for the $-1$ sign uncertainty principle, one might still hope to identify other regularity properties exhibited by near-minimizers. 
As previously mentioned, problem \eqref{eq:minA-} has been solved if $d\in\{1,8,24\}$. 
Moreover, if $d\in\{8,24\}$, then the minimizer is unique modulo symmetries, belongs to the Schwartz class, and is nonnegative in a neighborhood of the origin. 
If $d=1$, then uniqueness fails, but the nonnegative property still holds in all known examples; see Figure \ref{fig:best}.
In particular, one may expect every suitable function which is sufficiently close to a minimizer to concentrate a universal proportion of its positive mass on the smallest ball centered at the origin that contains all of its negative mass. Our second main result confirms these heuristics in all dimensions. Before stating it, recall (as shown in \cite{BCK10, CG19}) that we can restrict attention to radial functions. More precisely, the quantity $\mathbb{A}_s(d)$ coincides with the minimal value of $r(g)$ in the following optimization problem.

\begin{problem1}[$s$-FELPP]\label{prob:felp}
Let $s\in\{+,-\}$. Minimize $r(g)$ over all radial $g:\R^d \to \R$ such that
\begin{itemize}
\item $g\in L^1(\R^d)\setminus\{{\bf 0}\}$ and $\widehat g=sg;$ 
\item $g(0)=0$ and $g$ is eventually nonnegative.
\end{itemize}
\end{problem1}
Given $r>0$, let $B_r^d\subseteq\R^d$ denote the open ball of radius $r$ centered at the origin, and $g_+ := \max\{g,0\}$. 

\begin{theorem}\label{thm:minus-equiv} 
Given $d \ge 1$, there exist constants $\varepsilon_d,\sigma_d>0$, such that 
\begin{equation}\label{intt}
 \int_{B_{r(g)}^d} g_+ \ge \sigma_d\|g\|_{L^1(\R^d)},
\end{equation}
whenever $g\in\mathcal A_-(d)$ is a radial function such that $\widehat g=-g, g(0)=0$, and satisfies $|r(g) - \mathbb{A}_-(d)| \le \varepsilon_d$.
\end{theorem}

\begin{figure}[htbp]
  \centering
  \includegraphics[height=7cm]{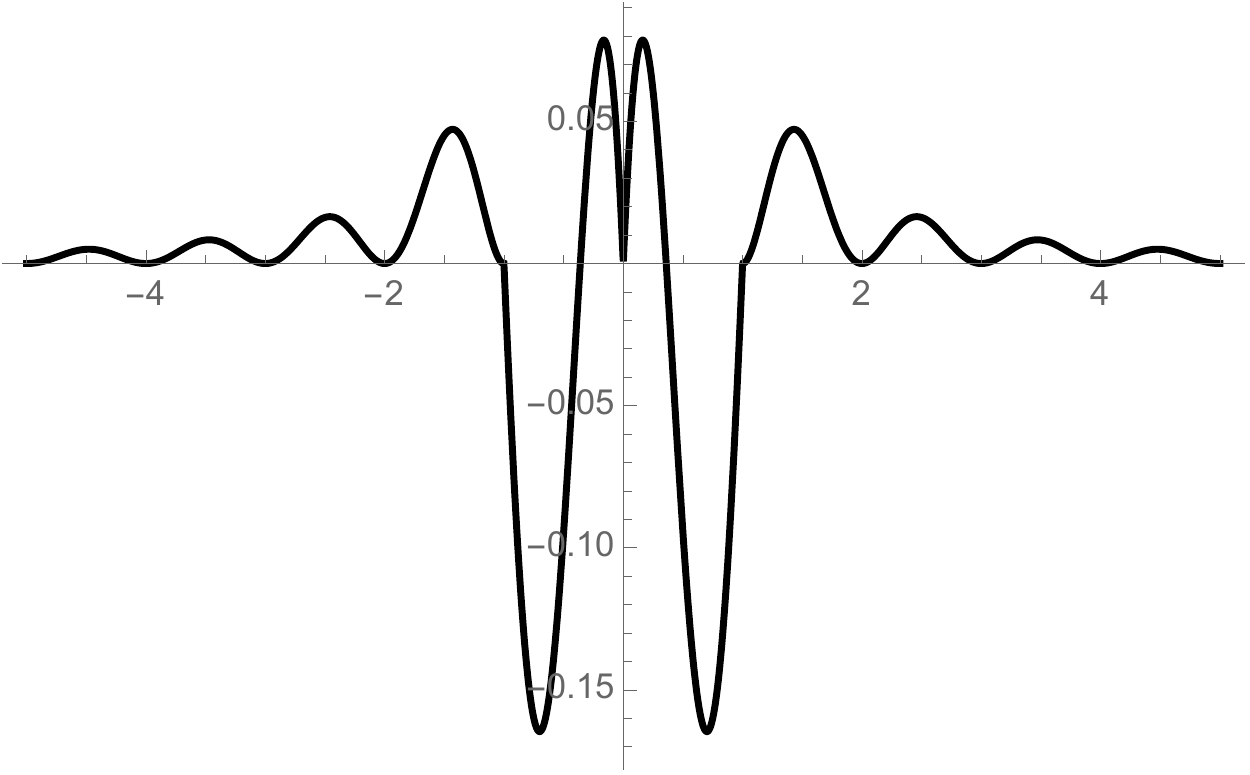} 
    \caption{Plot of  $g:=\widehat f-f$, where $f(x)=(1-|x|)_+$, $\widehat f(\xi)=\frac{\sin^2 (\pi \xi)}{(\pi \xi)^2}$.
    The function $g$ is a minimizer of the $-1$-FELPP when $d=1$.}
\label{fig:best}
\end{figure}

We conclude the Introduction with a remark connecting bandlimited functions, Poisson summation, and reconstruction formulae.
According to Theorem \ref{thm:plus-equiv}, $\mathbb{A}_+(1) = \mathbb{A}_+^{\mathcal B}(1)$. 
From this and the Poisson summation formula, a non-trivial conclusion can be  withdrawn as in \cite{L, L2}. 
We record it in the following result, 
as it may prove useful in further investigations surrounding the Fourier Eigenvalue Linear Programming Problem.

\begin{prop}\label{Poisson}
Let $f\in\mathcal A_+(1)$ be a nonzero, bandlimited function, such that \textup{supp}$(\widehat f)\subseteq [-\frac12,\frac12]$.
Then $r(f)\geq 1$.
Moreover, $r(f)=1$ if and only if there exists $\alpha>0$, such that 
\[f(x)=\alpha\frac{\sin^2(\pi \frac{x-1}2)}{x^2-1}.\]
\end{prop}

\subsection{Outline}
We prove Theorem \ref{thm:plus-equiv} in \S \ref{sec:main-proof}, 
 Theorem \ref{thm:minus-equiv} in \S \ref{sec:minus-proof},
 and Proposition \ref{Poisson} in \S \ref{sec:ProofProp1}. 
The arguments rely on several lemmata which we choose to formulate in general dimensions whenever possible, and prove in \S \ref{sec:PfLemmata}.
We dedicate the final \S \ref{sec:DiscreteVsContinuous} to a connection with the sign uncertainty principle for Fourier series on the torus, recently established in \cite{GOeSR19}.

\subsection{Notation}
Let $\N:=\{1,2,\ldots\}$ and $\N_0:=\N\cup\{0\}$.
Given $f:\R^d\to\R$, let $f_+:=\max\{f,0\}$, and
$f_-:=\max\{-f,0\}$.
In this way, $f_+,f_-$ are nonnegative functions which are never positive at the same point, and satisfy $f=f_+-f_-$ and $|f|=f_++f_-$.
Given a set $E\subseteq \R^d$, its indicator function is denoted by ${\mathbbm 1}_E$, and its Lebesgue measure by $|E|$.
Given $r>0$, we continue to let $B_r^d\subseteq\R^d$ denote the open ball of radius $r$ centered at the origin.

\section{Proof of Theorem \ref{thm:plus-equiv}}\label{sec:main-proof}

\subsection{Proof of $\mathbb{A}_+(1) = \mathbb{A}_+^{\mathcal B}(1)$}\label{ssec:BcBbl} 
It suffices to show that 
$\mathbb{A}_+(1) \ge \mathbb{A}_+^{\mathcal B}(1)$, and
we proceed in four steps. \\

\noindent\textit{Step 1.} \textit{Let $f\in\mathcal A_+(1)$ be a minimizer for \eqref{eq:+1-up} satisfying $\widehat{f}=f$, $f(0)=0$. 
Then there exists $\varepsilon>0$, such that $f(x)<0$, for every $x\in(-\varepsilon,\varepsilon)\setminus\{0\}$.} 

 The proof of Step 1 hinges on two distinct observations, the first of which is inspired by work of Logan \cite{L, L2} on various extremal problems concerning the behavior of positive-definite bandlimited functions. The following result holds for any sign $s\in\{+,-\}$ in arbitrary dimensions $d\geq 1$, and should be compared with \cite[Lemma]{L}, where a one-dimensional variant of the same result is proved.

\begin{lemm}\label{one} 
Given  $s\in\{+,-\}$, $d\geq 1$, let $f\in\A_s(d)$ be radial,  such that $\widehat f=s{f}$, $f(0) =0$. 
Suppose that there exists a sequence $\{x_n\}_{n\in\N}\subseteq\R^d\setminus\{0\}$, such that $x_n\to 0$, as $n\to\infty$, and $sf(x_n) \ge 0$, for all $n$. Then 
$ \int_{\R^d} |y|^2 |f(y)|\,\mmd y < +\infty$, 
and 
\begin{equation}\label{eq1}
  \int_{\R^d} |y|^2 f(y) \,\mmd y \le 0.
\end{equation}
\end{lemm}

 The second observation is that a chain of rearrangement-type inequalities can be set up in such a way as to contradict estimate \eqref{eq1}. 
A similar approach already proved fruitful in \cite[\S 4]{GOeSS17}.
To implement it in the present context, let $f\in\A_+(1)$ be a minimizer for \eqref{eq:+1-up}, which we suppose to be $L^1$-normalized, $\|f\|_{L^1}=1$, and  to satisfy $\widehat f={f}$, $f(0)=0$; see \cite[\S 3]{GOeSS17}. If $f$ fails to be strictly negative on any punctured neighborhood of the origin, then the hypotheses of Lemma \ref{one} for $s=+1$ are verified.  Writing $f=f_+-f_-$, we then have that
\begin{equation}\label{eq:rearrange} 
 \int_{0}^{\infty} y^2f(y) \,\mmd y = \int_{0}^{r(f)} y^2 f_+ (y) \,\mmd y - \int_{0}^{r(f)} y^2f_-(y) \,\mmd y + \int_{r(f)}^{\infty} y^2 f(y) \,\mmd y\leq 0.
\end{equation}
Setting $\sigma(f) := \|f_+\|_{L^1(0,r(f))}$, and appealing to \cite[Lemma 12]{GOeSS17}, we can bound each of the integrals on the right-hand side of \eqref{eq:rearrange} as follows:
\begin{align}
&\int_{0}^{r(f)} y^2 f_+(y) \,\mmd y \ge \int_{0}^{\sigma(f)} y^2 \,\mmd y; \label{eq:I1}\\
&\int_{0}^{r(f)} y^2 f_-(y)\, \mmd y \le \int_{r(f)-\frac{1}{4}}^{r(f)} y^2 \,\mmd y; \label{eq:I2}\\
&\int_{r(f)}^{\infty} y^2 f(y) \,\mmd y \ge \int_{r(f)}^{r(f)+\frac{1}{4} - \sigma(f)} y^2 \,\mmd y. \label{eq:I3}
\end{align}
To see why this is the case, note that $\|f\|_{L^\infty}\leq 1$, and that
\begin{align*}
\int_\R f_++\int_\R f_-&=\int_\R |f|=1;\\
\int_\R f_+-\int_\R f_-&=\int_\R f=\widehat f(0)=0,
\end{align*}
and thus $\int_\R f_+=\int_\R f_-=\frac12$;
moreover, the functions $f, f_\pm$ are even, and so their masses are equally spread over the positive and negative half-lines.
Estimates \eqref{eq:rearrange}--\eqref{eq:I3} immediately imply that
 \[  \frac{\sigma(f)^3}{3}  - \frac{r(f)^3 - (r(f)- \frac{1}{4})^3}{3}+ \frac{(r(f) + \frac{1}{4} - \sigma(f))^3 - \sigma(f)^3}{3}\leq 0,\]
which can be equivalently rewritten as
 \begin{equation}\label{import}
  \Big(r(f)+\frac{1}{4}\Big)\sigma(f) \Big(r(f)+\frac{1}{4} - \sigma(f)\Big) \ge \frac{r(f)}{8}.
 \end{equation}
We seek for a sufficiently good upper bound for $\sigma(f)$ which will then force the desired contradiction. 
With this purpose in mind, observe that the left-hand side of \eqref{import} defines an increasing function of $\sigma(f)$,  provided $\sigma(f) < \frac{1}{4}$. 
The next result provides a slight improvement over \cite[Lemma 14]{GOeSS17} in terms of the admissible radius $r(f)$. 

\begin{lemm}\label{two} 
Let $f\in\A_+(1)$ be such that  $\|f\|_{L^1}=1$, $\widehat f={f}$, $f(0) =0$, and $r(f)\in[\frac14,\frac1{\sqrt{2}}]$. 
Then the following inequality holds:
\begin{equation}\label{pointwisefplus}
f_+(x) \le \frac{1}{2} + \frac{\sin(2 \pi (r(f)-\frac{1}{4})x) - \sin(2 \pi r(f) x)}{\pi x},
\end{equation}
 for every $x \in [0,r(f)]$.
\end{lemm}

 With the two observations in place, we may now finish the proof of Step 1. 
Since $\|f\|_{L^\infty}\leq\|\widehat f\|_{L^1}=\|f\|_{L^1}= 1$ and $\|f_-\|_{L^1(0,r(f))}= \frac{1}{4}$,
the following superlevel set estimate holds:
\begin{equation}\label{eq:superlevel}
|\{x \in [0,r(f)]: f(x) \geq 0\}| \leq r(f) - \frac{1}{4}.
\end{equation} 
Since $f$ is a minimizer for \eqref{eq:+1-up}, and $\widehat f=f$, it follows from \cite[Theorem~2]{GOeSS17} that $0.45\leq r(f) \le 0.595$; in particular, $r(f)\in[\frac14,\frac1{\sqrt{2}}]$. 
As a consequence, we may appeal to \eqref{eq:superlevel} in order to estimate
\begin{align} 
\sigma(f) &= \int_{[0,r(f)]} f_+  
\le  \sup_{J \subseteq [0,r(f)] \atop |J| = r(f)- \frac14}  \int_J f_+ \label{chain1}\\
 & \le \int_{\frac{1}{4}}^{r(f)}  \left( \frac{1}{2} + \frac{\sin(2 \pi (r(f)-\frac{1}{4})x) - \sin(2 \pi r(f) x)}{\pi x} \right) \,\mmd x\notag
 :=\Phi(r(f)), 
\end{align}
where from the first to the second line we invoked Lemma \ref{two} and \cite[Lemma 11]{GOeSS17}, using the fact that the integrand on the right-hand side defines an increasing function of $x$.
We are thus reduced to a straightforward analysis of the function $\Phi$. 
It is easy to check that 
 $\Phi(r(f)) \leq\Phi(0.595)< 0.121$, and therefore the chain of inequalities \eqref{chain1} implies $\sigma(f) <  0.121$. 
But  if $\sigma(f) <  0.121$ and $0.45\leq r(f) \le 0.595$, then \eqref{import} does not hold.
This yields the desired contradiction, and concludes the proof of Step 1.\\

\noindent\textit{Step 2.} \textit{There exists a minimizing sequence $\{f_n\}_{n\in\N}\subseteq\mathcal A_+(1)$ for \eqref{eq:+1-up}, such that the following conditions hold, for every $n$:
\begin{itemize}
\item[(a)] $f_n(0)<0;$
\item[(b)] $f_n(x)>0$, for every $|x|> r(f_n);$
\item[(c)] $\widehat f_n={f}_n;$
\item[(d)] $x^2f_n(x)\geq c_n$, for some $c_n>0$ and  all sufficiently large $|x|$.
\end{itemize}} 
 Let $f\in\A_+(1)$ be a minimizer for $+1$-FELPP. By Step 1, there must exist a sequence $\{x_n\}_{n\in\N}\subseteq(0,\infty)$ such that $x_n\to 0$, as $n\to\infty$, and $f(x_n) < 0$, for every $n$. Define an associated sequence $\{T_n\}_{n\in\N}$ of tempered distributions  via
\[T_n := \delta_{x_n} + \delta_{-x_n} + 2\delta_0.\]
Setting $g_n := T_n * f$,  one easily checks that $\{g_n\}_{n\in\N}\subseteq\A_+(1)$ is a minimizing sequence for \eqref{eq:+1-up}. Indeed, the quantity
\[g_n(x) = f(x-x_n)+f(x+x_n)+2f(x)\]
is nonnegative if $x \ge r(f) + x_n$, and satisfies $g_n(0) = 2f(x_n) < 0$
(in particular, $g_n$ does not vanish identically). 
On the other hand, 
\[\widehat T_n(\xi) = 2 \cos(2\pi x_n \xi) + 2 \ge 0,\]
 and so $\widehat{g}_n(0) = 4\widehat{f}(0) \leq 0$.
By the scaling argument detailed in \cite[\S 3.3]{GOeSS17}, we lose no generality in assuming that $r(\widehat g_n)=r({g}_n)$. In this case,  $\{g_n+\widehat{g}_n\}_{n\in\N}\subseteq\A_+(1)$ is a minimizing sequence for \eqref{eq:+1-up} satisfying conditions (a) and (c). 
Condition (b) can be achieved by further adding a suitable Gaussian function: Setting 
\[h_n := g_n+\widehat{g}_n - \frac{g_n(0)+\widehat{g}_n(0)}{2} \exp(-\pi\cdot^2),\] 
one again checks that $\{h_n\}_{n\in\N}\subseteq\A_+(1)$ is a minimizing sequence for \eqref{eq:+1-up} which satisfies conditions (a)--(c). 
 In order to further ensure condition (d), we make use of the following simple observation.

\begin{lemm}\label{eta} 
Given $d\geq 1$, there exists a function $\eta \in \A_+(d)$, such that 
$\widehat\eta = \eta$,
 $\eta(0)< 0$,
and $|x|^{d+1}\eta(x) \geq  1$, for all sufficiently large values of $|x|$. 
\end{lemm}

 Let $\eta\in\A_+(1)$ be given by Lemma \ref{eta}, and pick the smallest $r_0>0$ such that $x^2\eta(x) \ge 1$, for every $|x| \ge r_0$. 
Given $n\in\N$, set $\beta_n=1$ whenever $r(h_n)\geq r_0$. 
Otherwise, given $\delta>0$ which is sufficiently small so that $r(h_n)+\frac{\delta}n<r_0$,  set $\beta_n=\beta_n(\delta)>0$ in such a way that 
\[h_n(x) +\beta_{n} \eta(x) > 0, \text{ for every }x\in\left[r(h_n) + \tfrac{\delta}{n}, r_0\right].\] 
That such a choice of $\beta_n$ is possible follows from the fact that each function $h_n$ is eventually (strictly) positive. 
Then the sequence $\{f_n:=h_n+\beta_n\eta\}_{n\in\N}\subseteq \mathcal A_+(1)$ is minimizing for \eqref{eq:+1-up}, and satisfies properties (a), (c), and (d). Letting $\delta\to 0^+$, we ensure that condition (b) is fulfilled as well.
This concludes the verification of Step 2.\\

\noindent{\it Step 3.} \textit{Let $f\in\A_+(1)$ be such that $\|f\|_{L^1} = 1$, $\widehat f = {f}$, $f(0)<0$.
Suppose that there exist constants $c,R>0$, such that
\begin{itemize}
\item[(i)] $f(x) > 0$, for every $|x| > r(f);$
\item[(ii)] $x^2f(x) \ge {c}$, for every $|x| \ge R.$
\end{itemize}
Then, given $\varepsilon >0,$ 
there exists a bandlimited function $g \in \A_+(1)$, such that $g(x) > 0$, for every $|x| \ge r(f)+\varepsilon$.} 
\vspace{.1cm}

Fix a nonnegative, even, compactly supported, smooth function  $\psi\in C_0^\infty(\R)$, such that  $\|\psi\|_{L^2}=1$. 
Set  $\varphi := \psi \ast \psi$ and  $\varphi_{\delta} (x) := \varphi(\delta x).$ 
As  $\delta\to0^+$, the family $\{\widehat\varphi_{\delta}\}_{\delta>0}$ constitutes an approximation to the identity.
Therefore the bandlimited function $g_\delta:=f * \widehat\varphi_{\delta}$ should provide a good approximation for $f$, for small enough values of $\delta$. 
We turn to the details.

Let $f,c,R$ be as above, and let $\varepsilon>0$ be arbitrary but given. 
We show that $\delta=\delta(f,c,R,\varepsilon)$ can be chosen sufficiently small, so that $g=g_\delta$ belongs to $\mathcal A_+(1)$, and satisfies $g(x)>0$, for every $|x|\geq r(f)+\varepsilon$.
Let $ R_1\geq R$ be such that 
\begin{equation}\label{esthatphi}
|\widehat{\varphi}(\xi)| \le 10^{-6}{c}{|\xi|^{-3}}, \text{ for every } |\xi|\ge R_1.
\end{equation} 
This is certainly possible since $\widehat{\varphi}$ is rapidly decreasing.
By letting $\delta\to0^+$, the difference $f-g_\delta$ can be made uniformly close to $0$ in the interval $[r(f)+\varepsilon,2R_1]$. Thus it suffices to consider $|x|>2R_1$, in which case the following chain of inequalities holds, as long as $\delta>0$ is sufficiently small:
\begin{align}
g_\delta(x)
& = \int_{-\infty}^\infty f(x-\xi)\widehat\varphi_{\delta}(\xi)\, \mmd \xi 
\ge \int_{-\delta^{1/2}}^{\delta^{1/2}}f(x-\xi)\widehat\varphi_{\delta}(\xi) \,\mmd \xi 
- \int_{x-r(f)}^{x+r(f)}\widehat\varphi_{\delta}(\xi) \,\mmd \xi \notag\\
& \ge \frac{c}{4x^2} \int_{-\delta^{1/2}}^{\delta^{1/2}}\widehat\varphi_{\delta}(\xi)\, \mmd \xi 
-\int_{x-r(f)}^{x+r(f)}\widehat\varphi_{\delta}(\xi) \,\mmd \xi
 \ge \frac{c}{8x^2} - \int_{x-r(f)}^{x+r(f)}\widehat\varphi_{\delta}(\xi) \,\mmd \xi.\label{bigest} 
\end{align}
In the first inequality, we used hypothesis (i) to ensure that $f(x-\xi)\geq 0$ unless $\xi\in[x-r(f),x+r(f)]$, and the fact that $\|f\|_{L^\infty}\leq 1$;
the second inequality holds for sufficiently small $\delta>0$, in view of hypothesis (ii);
and the third inequality holds for sufficiently small $\delta>0$ since $\{\widehat\varphi_{\delta}\}_{\delta>0}$ is an approximation to the identity.
Invoking \eqref{esthatphi}, we have that
\[0\leq \widehat\varphi_{\delta}(\xi) 
= \delta^{-1}\widehat{\varphi}(\delta^{-1}{\xi}) 
\le \frac{1}{\delta} \frac{c}{10^6 |\delta^{-1}{\xi}|^3} 
= \frac{c \delta^2}{10^6|\xi|^3},\]
for every $|\xi| \geq \delta R_1$. 
Since $|x|> 2R_1$, it follows that $|x-r(f)| \ge \delta R_1$, and therefore the last integral on the right-hand side of \eqref{bigest} can be estimated as follows:
\begin{equation}\label{smallest}
 \int_{x-r(f)}^{x+r(f)}\widehat\varphi_{\delta}(\xi) \,\mmd \xi 
\le \frac{c\delta^2}{10^6} \int_{x-r(f)}^{x+r(f)} \,\frac{\mmd \xi}{|\xi|^3} 
\le \frac{c/2}{10^6}\frac{\delta^2}{(x-r(f))^2}  
\le \frac{2 c}{10^6}\frac{\delta^2}{x^2} 
 \le \frac{c}{10^4x^2}, 
\end{equation}
provided $2\delta^2 < 10^2.$ 
Estimates \eqref{bigest}, \eqref{smallest} together imply that $g_\delta(x) >0$, for every $|x|>2R_1$, as long as $\delta>0$ is sufficiently small. 
Finally, one easily checks that $g_\delta\in\mathcal A_+(1)$, provided $\delta>0$ is sufficiently small.
For instance, $g_\delta(0)\leq 0$ since $f(0)<0$ (the strict inequality is crucial here) and $f$ is continuous. Step 3 follows.\\ 

\noindent{\it Step 4. Conclude.} Consider a minimizing sequence $\{f_n\}_{n\in\N}\subseteq\mathcal A_+(1)$ for \eqref{eq:+1-up}, satisfying $\|f_n\|_{L^1}=1$, $f_n(0)<0$, $f_n(x)>0$ if $|x|\geq r(f_n)$, $\widehat f_n=f_n$, and $x^2f_n(x)\geq c_n$ if $|x|$ is sufficiently large.
The existence of such a sequence follows from Step 2. 
Running the proof of Step 3 on each $f_n$ individually, we find that there exists a bandlimited function $g_n=f_n\ast\widehat\varphi_\delta\in\mathcal A_+(1)$, such that $r(g_n)\leq r(f_n)+\frac1n$.
Since $\widehat g_n=\widehat f_n\varphi_\delta$ has pointwise the same sign as $\widehat f_n$, 
we may let $n\to\infty$ and conclude that $\{g_n\}_{n\in\N}\subseteq\mathcal A_+(1)$ is a minimizing sequence for \eqref{eq:+1-up} consisting of bandlimited functions, as desired.\\

\subsection{Proof of $\mathbb{A}_+(1) = \mathbb{A}_+^{\mathcal{S}}(1)$}\label{ssec:Sch-equiv}

Following \cite[\S 1]{BCK10}, consider the restricted classes
\begin{align*}
\tilde{\mathcal{A}}_+(d)&:=\{f\in\mathcal{A}_+(d): \widehat f=f, f(0)<0\};\\
\tilde{\mathcal{A}}_+^{\mathcal{S}}(d)&:=\{f\in\mathcal{A}_+(d)\cap\Sch(\R^d): \widehat f={f}, f(0)<0\},
\end{align*}
and define the corresponding optimal constants
$$\tilde{\mathbb{A}}_+(d):=\inf_{f\in \tilde{\mathcal{A}}_+(d)\setminus\{{\bf 0}\}}\sqrt{r(f)r(\widehat{f})};
\;\;\;
\tilde{\mathbb{A}}_+^{\Sch}(d):=\inf_{f\in \tilde{\mathcal{A}}_+^{\Sch}(d)\setminus\{{\bf 0}\}}\sqrt{r(f)r(\widehat{f})}.$$
Our next result reveals that these constants coincide in all dimensions.

\begin{prop}\label{B-}
Let $d\geq 1$. Then $\tilde{\mathbb{A}}_+(d) = \tilde{\mathbb{A}}_+^{\Sch}(d)$.
\end{prop}

\begin{proof}[Proof of Proposition \ref{B-}]
It suffices to show that $\tilde{\mathbb{A}}_+(d) \geq \tilde{\mathbb{A}}_+^{\Sch}(d)$.
With this goal in mind, let $f\in\tilde{\mathcal{A}}_+(d)$. 
Given $\delta>0$, consider a nonnegative, radial, compactly supported, smooth function $\psi_\delta\in C_0^\infty(\R^d)$, such that
$\text{supp}(\psi_\delta)\subseteq B_\delta^d$;
further assume $\psi_\delta$ to be $L^1$-normalized, 
$\|\psi_\delta\|_{L^1}=1.$
Define $\varphi_\delta:=\psi_\delta\ast\psi_\delta$, $g:=f\ast\varphi_\delta$, and 
\begin{equation}\label{ghh_0}
 h:=\widehat{g}\ast\varphi_\delta+{g}\widehat\varphi_\delta.
\end{equation}
The following lemma lists the key properties of the function $h$.

\begin{lemm}\label{gh}
Let $f\in\tilde{\mathcal{A}}_+(d)$.
Given $\varepsilon>0$, there exists $\delta>0$, such that the function $h$ defined in \eqref{ghh_0} above satisfies the following properties:
\begin{itemize}
\item[(a)] $h\in\Sch(\R^d), \widehat h={h}, h(0)<0;$
\item[(b)] $r(h)\leq r(f)+\varepsilon$.
\end{itemize}
\end{lemm}

Part (a) of Lemma \ref{gh} implies that $h\in\tilde{\mathcal A}_+^{\Sch}(d)$, and 
the inequality $\tilde{\mathbb{A}}_+(d) \ge \tilde{\mathbb{A}}_+^{\Sch}(d)$ then follows from part (b) of Lemma \ref{gh} by letting $\varepsilon\to 0^+$.
\end{proof}

The desired identity, $\mathbb{A}_+(1) = \mathbb{A}_+^{\mathcal{S}}(1)$, follows from the chain of inequalities
\begin{equation}\label{chain}
\mathbb{A}_+(1) 
\leq \mathbb{A}_+^{\Sch}(1) 
\leq \tilde{\mathbb{A}}_+^{\Sch}(1)
\leq\tilde{\mathbb{A}}_+(1)
\leq \mathbb{A}_+(1).
\end{equation}
The first two inequalities in \eqref{chain} are trivial, and the third one follows from Proposition \ref{B-}. 
Our next result addresses the last inequality in \eqref{chain}.

\begin{prop}\label{prop:tildenotilde}
$\tilde{\mathbb{A}}_+(1) \leq \mathbb{A}_+(1)$.
\end{prop}
While the argument in Proposition \ref{B-} could be adapted to handle the present case as well, we offer a different proof which does not hinge on Lemma \ref{gh} and could prove of independent interest.
\begin{proof}[Proof of Proposition \ref{prop:tildenotilde}]
Let $f\in{\mathcal{A}}_+(1)$ be a minimizer for \eqref{eq:+1-up} satisfying $\widehat f=f$, $f(0)=0$; 
in particular, $r(f)=\mathbb{A}_+(1)$.
Given $\varepsilon>0$, we can invoke  Step 1 of \S \ref{ssec:BcBbl} in order to ensure that 
there exists $x_0\in (-\varepsilon,\varepsilon)$, such that $f(x_0)<0$. 
Similarly to Step 2 of \S \ref{ssec:BcBbl},  consider the measure 
$T:=\delta_{x_0}+\delta_{-x_0}+2\delta_0,$
which is positive-definite in the sense that
\begin{equation}\label{eq:hatTpos}
\widehat{T}(\xi)=2\cos(2\pi x_0\xi)+2\geq 0,
\end{equation}
for every $\xi$, and define $g:=T\ast f$. Since
\begin{equation}\label{gintermsoff}
g(x)=f(x-x_0)+f(x+x_0)+2f(x),
\end{equation}
 it follows that $g$ is real-valued, integrable, and even. Moreover, 
$g(0)=2f(x_0)<0.$
Since $|x_0|\leq \varepsilon$, it follows from \eqref{gintermsoff} that 
\begin{equation}\label{ggeq0}
g(x)\geq 0, \text{ if } |x|\geq r(f)+\varepsilon.
\end{equation} 
In light of \eqref{eq:hatTpos}, we have that
$\widehat{g}(\xi)=\widehat{T}(\xi)\widehat{f}(\xi)\geq 0$
 if and only if 
 $\widehat{f}(\xi)\geq 0$.
In particular, 
\begin{equation}\label{hatggeq0}
\widehat{g}(\xi)\geq 0, \text{ if } |\xi|\geq r(f).
\end{equation} 
Moreover,
$\widehat{g}(0)=4f(0)=0.$
Consider the dilation $h:=g(\lambda \cdot)$, where $\lambda:=(1+\varepsilon/ r(f))^{1/2}$. 
Then $h(0)=g(0)<0$, $\widehat{h}(0)=\lambda^{-1}\widehat{g}(0)=0$,
and the functions $h, \widehat{h}$ are nonnegative on the interval $[{(r(f)^2+\varepsilon r(f))^{1/2}},\infty)$. 
Real-valuedness, integrability, and evenness of $h,\widehat h$ follow from those of $g$. 
Therefore $h+\widehat{h}\in\tilde{\mathcal{A}}_+(1)$, whence $\tilde{\mathbb{A}}_+(1)\leq (r(f)^2+\varepsilon r(f))^{1/2}$.
Letting $\varepsilon\to 0^+$, it follows that $\tilde{\mathbb{A}}_+(1)\leq r(f)=\mathbb{A}_+(1)$. 
\end{proof} 

\subsection{A short proof of the higher dimensional improvement \eqref{eq:improvement}}\label{sec:weak-proof} 
Let $f\in \mathcal A_+(d)$ be a nonzero function, for which the product $r(f)r(\widehat f)$ is sufficiently close to $\mathbb{A}_+(d)^2$.
By the usual reductions, we may assume that the function $f$ is radial, and satisfies $f = \widehat{f},\, f(0) = 0$, and $\|f\|_{L^1} = 1.$ 
In particular,
\[|\{x \in B_{r(f)}^d \colon f(x) < 0\}| \ge \int_{B_{r(f)}^d} f_- = \tfrac{1}{2}.\]
This implies the following superlevel set estimate:
\begin{equation}\label{eq:bound-positive}
| \{ x \in B_{r(f)}^d \colon f(x) \geq 0 \}| \le \nu_d r(f)^d - \tfrac{1}{2},
\end{equation}
where, as usual, $\nu_d=2d^{-1}\pi^{\frac d2}\Gamma(\frac d2)^{-1}$ is the Lebesgue measure of the unit ball in $\R^d$. 
Suppose that  $r>0$ is such that $f(x) \ge 0$, for every $x \in B_r^d$. 
Trivially, $r<r(f)$. 
In order to improve on this, we simply note that \eqref{eq:bound-positive} implies
$r^d \le  r(f)^d - (2 \nu_d)^{-1}$.
Thus there exists $x_0 \in \R^d$, satisfying 
\[|x_0| \le \left( r(f)^d - (2 \nu_d)^{-1}\right)^{\frac1d},\text{ and }f(x_0) < 0.\]  
The rest of the argument follows very similar ideas to those in \S \ref{ssec:Sch-equiv}. 
In particular, 
we obtain the following inequality: 
\begin{equation}\label{eq:bound-diff}
\mathbb{A}_+^{\Sch}(d) \le \left( 1 + \left( 1 - \frac{1}{2\nu_d {\mathbb{A}_+(d)}^d}\right)^{\frac1d} \right) \mathbb{A}_+(d).
\end{equation}
Setting $\theta_d := (2\nu_d)^{-1} \mathbb{A}_+(d)^{-d}$, we  estimate 
\[1 - \left( 1- \theta_d \right)^{\frac1d} = \frac{1}{d} \int_{1-\theta_d}^1 t^{\frac1d - 1} \, \mmd t \ge \frac{\theta_d}{d}.\] 
Stirling's formula and  \eqref{eq:lim-inf-sup-plus}  together imply that
$\theta_d \geq C \sqrt{d} (2\pi e)^{-d/2}$, for some absolute constant $C>0$.
Plugging this back into \eqref{eq:bound-diff} yields \eqref{eq:improvement}.

\section{Proof of Theorem \ref{thm:minus-equiv}}\label{sec:minus-proof} 
The proof of Theorem \ref{thm:minus-equiv} proceeds by contradiction. 
We start by establishing the following claim.
{\it If Theorem \ref{thm:minus-equiv} does not hold, then there exists a radial minimizer $f\in\mathcal A_-(d)$ of \eqref{eq:minA-}, satisfying $\widehat f=-f$, $f(0)=0$,  such that $f(x)\leq 0$ for all $x\in B_{r(f)}^d$.}

To see why this is necessarily the case, start by observing that, if Theorem \ref{thm:minus-equiv} does not hold, then there exists a minimizing sequence $\{f_n\}_{n\in\N}\subseteq \mathcal A_-(d)$ for \eqref{eq:minA-} consisting of radial functions, satisfying
$\widehat f_n=-f_n$, 
$f_n(0)=0$, 
$\|f_n\|_{L^1} = 1$,   
 $|r(f_n) - \mathbb{A}_-(d)| \leq \frac1n$, and
\begin{equation}\label{eq:NotThm2}
\int_{B_{r(f_n)}^d} (f_n)_+ \le \frac{1}{n}.
\end{equation}
No generality is lost in assuming that the sequence $\{r(f_n)\}_{n\in\N}$ is strictly decreasing.
By Jaming's higher dimensional version of Nazarov's uncertainty principle \cite{Ja07}, 
there exists a constant $K_d>0$ such that, for every $n \in\N$,
\begin{equation}\label{eq:NazarovJaming}
\int_{B_{r(f_n)}^d} f_n \leq -K_d.
\end{equation}
In fact, this amounts to a  straightforward modification of \cite[Lemma 23]{GOeSS17}, as was already observed in \cite[\S 3.2]{CG19}.
Since $\widehat f_n = - {f_n}$ and $\|f_n\|_{L^1} = 1,$ it follows that 
\[\|f_n\|_{L^2}^2\leq \|f_n\|_{L^1}\|f_n\|_{L^\infty}\leq \|f_n\|_{L^1}\|\widehat f_n\|_{L^1}=\|f_n\|_{L^1}^2=1.\] 
The Banach--Alaoglu Theorem then implies, possibly after extraction of a subsequence, that the sequence $\{f_n\}_{n\in\N}$ converges weakly to some function $f\in L^2$. 
Since $\mathcal A_-(d)$ is convex, we can then apply Mazur's Lemma to produce a sequence $\{g_n\}_{n\in\N}$ which converges strongly to $f$ in $L^2$, with each $g_n$ being a finite convex combination of elements from $\{f_m\}_{m\geq n}$.
For each $n \in\N$, $g_n$ is radial, $\widehat g_n=-g_n$, $g_n(0)=0$, $\|g_n\|_{L^2} \le 1$, and, possibly after passing to a further subsequence, the following additional properties hold:
\begin{align}
  &|r(g_n)- \mathbb{A}_-(d)| \leq \frac1n;\label{eq:P1}\\
  &\int_{B_{r(g_n)}^d} (g_n)_+ \le \frac{1}{n};\label{eq:P2}\\
  &\int_{B_{r(g_n)}^d} g_n \le -\frac{K_d}{2};\label{eq:P3}\\
  &g_n \to f \text{ almost everywhere, as } n\to\infty.\label{eq:P4}
\end{align}
Property \eqref{eq:P1} follows from the fact that the sequence $\{r(f_n)\}_{n\in\N}$ is decreasing, together with each $g_n$ being a convex combination of elements from the {\it tail} sequence $\{f_m\}_{m\geq n}$. 
Properties \eqref{eq:P2}, \eqref{eq:P3} follow from the latter observation, together with the elementary inequality  $(x+y)_+ \le x_+ + y_+$, valid for every $x,y \in \R$, and estimates \eqref{eq:NotThm2}, \eqref{eq:NazarovJaming}. 
Property \eqref{eq:P4} follows at once by extracting a further subsequence. 
In particular, from \eqref{eq:P3} it follows that 
\[\int_{B_{r(f)}^d} f \le - \frac{K_d}{4},\]
and therefore $f$ does not vanish identically.
Moreover,  \eqref{eq:P2} implies that
\[\int_{B_{r(f)}^d} f_+ = 0,\]
hence $f \le 0 \text{ in } B_{r(f)}^d$.
To conclude the proof of the claim, we still have to check that $f$ satisfies all required properties.
From \eqref{eq:P1} it follows that $r(f)=\mathbb A_-(d)$.
From \eqref{eq:P4} it follows that $f$ is radial.
That $\widehat f=-f$ follows at once from weak convergence.
Setting $r_1:=r(f_1)=\sup_{n\in\N} r(f_n)$, we may 
apply Fatou's Lemma to the sequence $\{g_n +   \mathbbm 1_{B_{r_1}^d}\}_{n\in\N}$ (consisting of nonnegative functions which converge almost everywhere to $f+  \mathbbm 1_{B_{r_1}^d}$) and deduce that $f\in L^1$, and that 
\[\widehat f(0)=\int_{\R^d} f \le \liminf_{n\to\infty} \int_{\R^d} g_n = 0.\]
In particular, ${f}(0) = - \widehat f(0) \geq 0$. 
Since $f \le 0 \text{ in } B_{r(f)}^d$, it follows that $f(0) = 0$. 
This concludes the verification of the claim.\\ 

We will need a higher dimensional version of the rearrangement inequalities used in \S \ref{ssec:BcBbl}. 
The following elementary result from \cite[\S 1.14]{LL} suffices for our application.

\begin{lemm}[Bathtub Principle, \cite{LL}]\label{thm:bathtubp} 
Let $h: \R^d \to \R$ be a measurable function, such  that $|\{x \in \R^d\colon h(x) < t\}|$ is finite for all $t \in \R$. 
Let the number $G >0$ be given, and define a class of measurable functions on $\R^d$ by
\[
\mathcal C_G = \left\{ g: \R^d \to [0,1]: \int_{\R^d} g = G\right\}.
\]
Then the minimization problem 
\begin{equation}\label{intfg}
I=\inf_{g\in\mathcal C_G}\int_{\R^d} gh
\end{equation}
is solved by $g(x) =    \mathbbm 1_{\{h < s\}}(x) + c     \mathbbm 1_{\{h = s\}}(x),$ where 
\begin{align*}
s &= \sup \{ t \in \R \colon |\{x\in\R^d: h(x) < t\}| \le G\};\\
c |\{x \in \R^d\colon h(x) = s\}| &= G - |\{ x \in \R^d \colon h(x)<s\}|.
\end{align*}
The minimizer is unique if 
$G=|\{ x \in \R^d \colon h(x)<s\}|$ or if 
$G=|\{ x \in \R^d \colon h(x)\leq s\}|$.
\end{lemm}
 
If Theorem \ref{thm:minus-equiv} does not hold, then we have already verified the existence of a radial minimizer $f\in\mathcal A_-(d)$ of \eqref{eq:minA-}, satisfying $\widehat f=-f$, $f(0)=0$,  such that $f\leq 0$ on $B_{r(f)}^d$.
Normalizing $\|f\|_{L^1} = 1$, so that $\|f\|_{L^\infty} \le 1$, we then have that
\begin{equation}\label{eq:inout1/2}
-\int_{B_{r(f)}^d} f = \int_{\R^d \backslash B_{r(f)}^d} f = \frac{1}{2}.
\end{equation}
Invoking Lemma \ref{one} with $s=-1$, and then Lemma \ref{thm:bathtubp} with $h(y)=|y|^2$, we have that
\begin{align}
0\geq\int_{\R^d} |y|^2 f(y) \, \mmd y 
 &= \int_{\R^d \backslash B_{r(f)}^d} |y|^2 f(y) \, \mmd y +  \int_{B_{r(f)}^d} |y|^2 f(y) \, \mmd y \label{eq:negtocontr1}\\
 &\ge  \int_{B_{t}^d\setminus B_{r(f)}^d} |y|^2\, \mmd y - \int_{B_{r(f)}^d\setminus B_{s}^d} |y|^2 \, \mmd y \label{eq:negtocontr2}\\
 &= \frac{\nu_d}{1+2/d} (t^{d+2} + s^{d+2} - 2 r(f)^{d+2}). \label{eqq1}
\end{align}
Here, $\nu_d :=2d^{-1}\pi^{\frac d2}\Gamma(\frac d2)^{-1}$ denotes the Lebesgue measure of the unit ball $B_1^d\subseteq\R^d$ and, in light of \eqref{eq:inout1/2}, the radii $s<t$ in \eqref{eq:negtocontr2} are defined in such a way as to ensure that 
$|B_{r(f)}^d\setminus B_{s}^d| = |B_{t}^d\setminus B_{r(f)}^d| =\frac{1}{2}$.
Equivalently, 
\[s = \left(r(f)^d - \frac{1}{2\nu_d}\right)^{1/d},\; t = \left(r(f)^d + \frac{1}{2\nu_d}\right)^{1/d}.\]
The strict convexity of the function $x \mapsto |x|^{1 + \frac{2}{d}}$ implies that
\[{t^{d+2} + s^{d+2}} > 2\left(\frac{t^d + s^d}{2}\right)^{\frac{d+2}{d}} = 2r(f)^{d+2},\]
hence \eqref{eqq1} defines a positive quantity.
The chain of inequalities \eqref{eq:negtocontr1}--\eqref{eqq1} then leads to the desired contradiction.
This concludes the proof of  Theorem \ref{thm:minus-equiv}.


\section{Proof of Proposition \ref{Poisson}}\label{sec:ProofProp1}

Let $f\in\mathcal A_+(1)\setminus\{{\bf 0}\}$ be a bandlimited function, such that \textup{supp}$(\widehat f)\subseteq [-\frac12,\frac12]$.
With the Fourier transform normalized as in \eqref{eq:FTnorm}, the Poisson summation formula implies
\[\sum_{n\in\Z} f(n)=\sum_{k\in\Z}\widehat{f}(k).\]
Dilating by a parameter $\alpha>0$, and translating by an arbitrary $\beta\in\R$, yields
\begin{equation}\label{dilPoisson}
\alpha\sum_{n\in\Z} f(\alpha n+\beta)=\sum_{k\in\Z}\widehat{f}\left(\tfrac k{\alpha}\right) \exp\left(2\pi i \tfrac{\beta}{\alpha} k\right).
\end{equation}
The function $\widehat{f}$ is continuous, and so $\widehat{f}(\pm\frac 12)=0$. 
Thus the right-hand side of \eqref{dilPoisson} equals $\widehat{f}(0)$ provided $\alpha\leq 2$,
and so
\begin{equation}\label{eq:AfterPoisson}
\sum_{n\in\Z} f(\alpha n+\beta)=\alpha^{-1}\widehat{f}(0)\leq 0,
\end{equation}
for every $(\alpha,\beta)\in (0,2]\times\R$. 
Aiming at a contradiction, suppose  $r(f)<1$. 
Given $\alpha\in (0,2]$ so that $r(f)\leq\frac{\alpha}2\leq1$, set $\beta=\frac{\alpha}2$. 
From \eqref{eq:AfterPoisson}, it follows that 
\begin{equation}\label{zerosum}
\sum_{n\in\Z} f((2n+1)\beta)\leq 0,
\end{equation}
for every $\beta\in[r(f),1]$.
But $f(x)\geq 0$, for every $|x|\geq r(f)$, and so \eqref{zerosum} implies
$$f(x)=0,\text{ for every } x\in[r(f),1].$$
Being the Fourier transform of a compactly supported function,  the function $f$ extends to an entire function on the whole complex plane, and as such it cannot vanish on a non-degenerate interval without being identically zero. 
This shows that $r(f)\geq 1$.

Now, suppose $r(f)=1$.
Then replicating the argument from the previous paragraph with $\beta=\frac{\alpha}2=1$ yields  
\begin{equation}\label{cond1}
f(2n+1)=0,\text{ for every } n\in\Z.
\end{equation}
Since the function $f$ does not change signs at the zeros $\pm 3, \pm 5,\ldots$, we also have that
\begin{equation}\label{cond2}
f'(2n+1)=0,\text{ for every } n\in\Z\setminus\{-1,0\}.
\end{equation}
Set $g(x):=f(2x+1)$. Then $g\in L^2(\R)$, and $\widehat{g}$ is supported on $[-1,1]$. 
By the Paley--Wiener--Schwarz Theorem \cite{Sc52}, the function $g$ coincides with the restriction to $\R$ of a complex-valued entire function of exponential type $2\pi$. It follows from Vaaler's interpolation formula \cite[Theorem 9]{V} that
$$g(x)=\left(\frac{\sin \pi x}{\pi}\right)^2\left(\sum_{m\in\Z}\frac{g(m)}{(x-m)^2}+\sum_{n\in\Z}\frac{g'(n)}{x-n}\right),$$
where the expression on the right-hand side converges uniformly on compact subsets of the real line.
Conditions \eqref{cond1}, \eqref{cond2} then translate into 
$$g(x)=\left(\frac{\sin \pi x}{\pi}\right)^2\left(\frac{g'(-1)}{x+1}+\frac{g'(0)}{x}\right),$$
which can be rewritten in terms of the original function $f$ as
$$f(2x+1)=\left(\frac{\sin \pi x}{\pi}\right)^2\left(\frac{2f'(-1)}{x+1}+\frac{2f'(1)}{x}\right).$$
Since $f$ is even, $f'(-1)=-f'(1)$. Consequently,
$$f(2x+1)=2f'(1)\left(\frac{\sin \pi x}{\pi}\right)^2\frac{1}{x(x+1)}.$$
A change of variables yields 
\begin{equation}\label{bestf}
f(x)=\frac{8f'(1)}{\pi^2}\frac{\sin^2(\pi \frac{x-1}2)}{x^2-1},
\end{equation}
which belongs to the class $\mathcal{A}_+(1)$ if and only if $f'(1)\geq 0$.
For the converse direction, note that the Fourier transform of the function $f$ given by \eqref{bestf}  is easy to calculate:
\[\widehat{f}(\xi)=\frac{4f'(1)}{\pi}({  \mathbbm  1}_{[-\frac 12,0]}-{ \mathbbm  1}_{[0,\frac 12]})(\xi)\sin(2\pi \xi).\]
In particular, this shows that $f$ is a bandlimited function with Fourier support $[-\frac12,\frac12]$; see Figure \ref{fig:best2}. 
It is  easy to check that $f\in\mathcal{A}_+(1)$, and that $r(f)=1$.
This concludes the proof of Proposition \ref{Poisson}.

\begin{figure}[htbp]
  \centering
  \includegraphics[height=7cm]{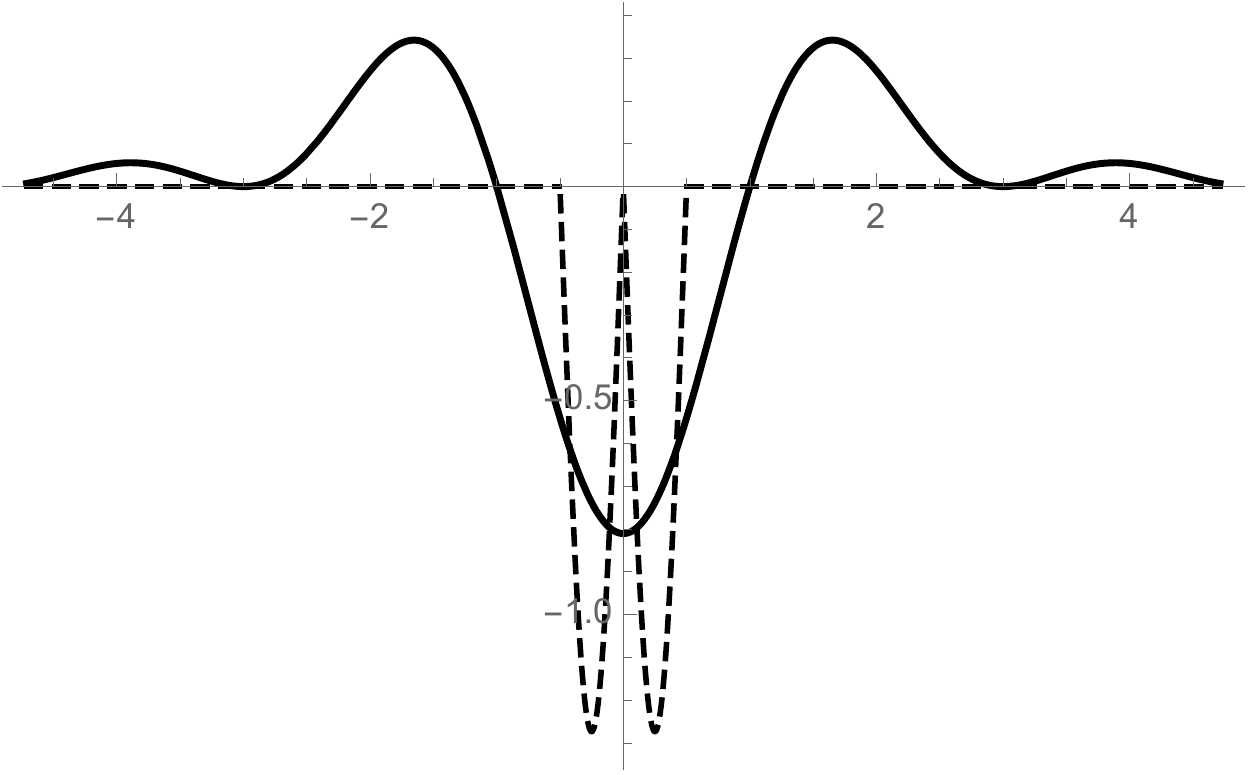} 
    \caption{Plot of the function $f(x)=\frac{8}{\pi^2}\frac{\sin^2(\pi \frac{x-1}2)}{x^2-1}$ (thick), and of its Fourier transform $\widehat{f}(\xi)=\frac{4}{\pi}({ \mathbbm 1}_{[-\frac 12,0]}-{ \mathbbm 1}_{[0,\frac 12]})(\xi)\sin(2\pi \xi)$ (dashed).}
\label{fig:best2}
\end{figure}


\section{Proofs of Lemmata}\label{sec:PfLemmata}

\begin{proof}[Proof of Lemma \ref{one}]
Let $s\in\{+,-\}$ and $f\in\mathcal A_-(d)$ be as in the statement of the lemma.
Since $\widehat f= sf$, $f$ is even, and $sf(x_n)\geq 0=f(0)$, we have that
\begin{align}\label{LoganEst}
\begin{split}
&  \frac{f(0) - sf(x_n)}{|x_n|^2} 
\\ & = \int_{B_{r(f)}^d} \frac{1-\cos(2\pi \langle x_n, y\rangle)}{|x_n|^2} f(y) \,\mmd y + \int_{\R^d\setminus B_{r(f)}^d} \frac{1-\cos(2 \pi \langle x_n, y\rangle)}{|x_n|^2} f(y)\, \mmd y \leq 0.
\end{split}
\end{align}
Uniformly on compact subsets of the real line, 
\begin{equation}\label{lim}
 \lim_{t \to 0} \frac{1-\cos(2 \pi t u)}{t^2} = 2\pi^2 u^2.
 \end{equation}
It follows that the first summand on the right-hand side of \eqref{LoganEst} is bounded independently of the sequence $\{x_n\}_{n\in\N}$, and tends to a finite limit, as $n\to\infty$. 

 Let $e_k\in\R^d$ denote the $k$-th coordinate unit vector, and write $y=(y_1,\ldots,y_d)\in\R^d$. 
 Since $f$ is radial, we lose no generality in assuming that $x_n=\lambda_ne_k$, for some $\lambda_n>0$ and $1\leq k\leq d$.
 Then $|x_n|=\lambda_n$ and $\langle x_n,y\rangle=\lambda_n y_k$.
Given $R\geq r(f)$, the second summand on the right-hand side of \eqref{LoganEst} (whose integrand is nonnegative on the region of integration) can be bounded from below as follows:
\[ \int_{\R^d\setminus B_{r(f)}^d} \frac{1-\cos(2 \pi \langle x_n, y\rangle)}{|x_n|^2} f(y)\, \mmd y 
\ge \int_{B_R^d\setminus B_{r(f)}^d} \frac{1-\cos(2 \pi \langle x_n, y\rangle)}{|x_n|^2} f(y) \,\mmd y.\] 
It then follows from \eqref{LoganEst}, \eqref{lim} that
\[ \sup_{R \geq r(f)} \int_{B_R^d\setminus B_{r(f)}^d} y_k^2 f(y) \,\mmd y < + \infty,\] 
whence $ \int_{\R^d} |y|^2 |f(y)| \,\mmd y < + \infty.$ 
Finally, \eqref{eq1} follows from noting that
\[\int_{\R^d} {y_k^2} f(y) \,\mmd y
=\frac 1{2\pi^2}\lim_{n\to\infty} \int_{\R^d} \frac{1-\cos(2\pi \langle x_n, y\rangle)}{|x_n|^2} f(y) \,\mmd y\leq 0,\]
and summing in $k$. 
\end{proof}

\begin{proof}[Proof of Lemma \ref{two}] 
Reasoning as in the proof of \cite[Th\'eor\`eme 1.1]{BCK10}, we have that 
\begin{equation}\label{estone}
f_+(x) \le 2\int_{0}^{\infty} f_-(y) (1-\cos(2\pi x y))\, \mmd y.
\end{equation}
If  $x \in[0,r(f)]$, then $y \mapsto 1- \cos(2 \pi x y)$ defines a nonnegative, increasing function of $y$ on the interval $[0,r(f)]$, provided $r(f)\leq\frac{1}{\sqrt{2}}$.
Since supp$(f_-)\subseteq [-r(f),r(f)]$, and $\|f_-\|_{L^1(0,r(f))}=\frac14\leq r(f)$, by \cite[Lemma 11]{GOeSS17} it then follows that
\begin{align}
\int_{0}^{\infty} f_-(y) (1-\cos(2\pi x y))\, \mmd y 
&\le \int_{r(f)-\frac{1}{4}}^{r(f)} (1- \cos(2 \pi x y)) \,\mmd y\notag\\
&= \frac{1}{4} +  \frac{\sin(2 \pi (r(f)-\frac{1}{4})x) - \sin(2 \pi r(f) x)}{2\pi x}.\label{esttwo}
\end{align}
Estimate \eqref{pointwisefplus}  follows at once from \eqref{estone}, \eqref{esttwo}, and the lemma is proved.
\end{proof}

\begin{proof}[Proof of Lemma \ref{eta}]
Let $  \mathbbm 1_{B_1^d}$ denote the indicator function of the unit ball $B_1^d\subset\R^d$, and define the convolutions $\chi :=   \mathbbm 1_{B_1^d}\ast   \mathbbm 1_{B_1^d}$ and $\varphi := \chi \ast \widehat{\chi}$.  Explicitly, 
\[\chi(x)=(2-|x|)_+,\;\;\;
 \widehat{\chi} (\xi) = \left(\frac{J_{d/2}(2\pi |\xi|)}{|\xi|^{d/2}}\right)^2,\]
 where $J_{d/2}$ denotes the Bessel function of the first kind.
 One easily checks that $\varphi,\widehat{\varphi}\in L^1$.
Moreover, the function ${\varphi}$ is bandlimited, and standard decay properties of the Bessel functions imply the existence of a constant $A>0$, such that
 \[\varphi(x)=\int_{B_2^d} (2-|t|)\frac{J_{d/2}^2(2\pi|x-t|)}{|x-t|^d} \,\d t\geq  A{|x|^{-(d+1)}},\] 
provided $|x|$ is sufficiently large. Let $\psi := \varphi + \widehat{\varphi}$. 
Then $\widehat\psi= \psi$, and $|x|^{d+1}\psi(x) \geq A$, for all sufficiently large values of $|x|$. This is still not the desired function since $\psi(0)>0$, but one can simply add an appropriate Gaussian function. For instance, the function 
$x\mapsto\eta(x) :=A^{-1}(\psi(x)-2\psi(0)\exp(-\pi |x|^2))\in\A_+(d)$ satisfies all the required properties.
\end{proof}

\begin{proof}[Proof of Lemma \ref{gh}]
If $\delta>0$ is small enough, then 
\begin{equation}\label{eq:lessthan0}
g(0)=\int_{B_{2\delta}^d}f\varphi_\delta<0,
\end{equation}
since the function $f$ is continuous and $f(0)<0$.
Moreover, $\widehat{g}(0)=\widehat{f}(0)(\widehat\psi_\delta(0))^2<0$.
Further note that
 $g(x)\geq 0$, provided $f\geq 0$ on the ball $x+B_{2\delta}^d$, and so it suffices to take $\delta\leq\frac{\varepsilon}2$ in order to ensure that $r(g)\leq r(f)+\varepsilon$. 
With these preliminary observations in place, we are now ready for the proof of the lemma.
 
For part (a), one easily checks that $h$ is a Schwartz function which coincides with its own Fourier transform. 
Moreover, 
\[h(0)=(\widehat{g}\ast\varphi_\delta)(0)+g(0)(\widehat\psi_\delta(0))^2,\]
where both summands are negative. 
For the first summand, note that
$(\widehat{g}\ast\varphi_\delta)(0)=\int_{B_{2\delta}^d} \widehat{g}\varphi_\delta$ is negative
if $\delta>0$ is small enough, since the function $\widehat{g}$ is continuous and $\widehat{g}(0)<0$. 
For the second summand, this is clear in light of \eqref{eq:lessthan0} and the real-valuedness of $\widehat\psi_\delta$. 

For part (b), we seek to verify that
$h(x)\geq 0, \text{ if } |x|\geq r(f)+\varepsilon$,
which will follow from
\[(\widehat{g}\ast\varphi_\delta)(x)\geq 0 \text{ and } {g}(x)\widehat\varphi_\delta(x)\geq 0,\text{ if } |x|\geq r(f)+\varepsilon.\]
The lower bound for ${g}\widehat\varphi_\delta$ follows immediately from $r(g)\leq r(f)+\varepsilon$,
whereas $\widehat f={f}$ implies
\[\widehat{g}\ast\varphi_\delta=(\widehat{f}\cdot\widehat\varphi_\delta)\ast\varphi_\delta=(f\cdot\widehat\varphi_\delta)\ast\varphi_\delta.\]
But $r(f\cdot\widehat\varphi_\delta)\leq r(f)$ since $\widehat\varphi_\delta=(\widehat\psi_\delta)^2\geq 0$, and so the lower bound for $\widehat{g}\ast\varphi_\delta$ follows as before. This concludes the verification of part (b), and the proof of the lemma.
\end{proof}


\section{From continuous to discrete uncertainty}\label{sec:DiscreteVsContinuous} 

Very recently, the authors \cite{GOeSR19} established  new and very general sign uncertainty principles which apply to certain classes of bounded linear operators and metric measure spaces. 
As a corollary, we obtained a sign uncertainty principle for Fourier series on the $d$-torus $\mathbb T^d=\R^d/\Z^d$, which we proceed to describe.

Given $s\in\{+,-\}$ and $d\geq 1$, let $\mathcal{P}_s(\mathbb{T}^d)$ denote the class of continuous, even  functions $g:\mathbb T^d\to\R$, such that $\widehat g\in \ell^1$, $\widehat g(0)\leq 0$, and
 $s \widehat g$ is eventually nonnegative while $sg(0)\leq 0$.
Given $g\in\mathcal P_s(\mathbb T^d)$, define the quantities\footnote{Trivially, $r(g;\mathbb T^d)\leq \tfrac{\sqrt{d}}2$ for all continuous functions $g:\mathbb T^d\to\mathbb{R}$.}
\begin{align*}
r(g;\mathbb T^d) &:= \inf\{r >0: g(x)\geq 0 \text{ if }\tfrac{\sqrt{d}}{2}\geq |x|_2\geq  r\};\\
k_s(\widehat{g}) &:= \min\{k \ge 1 \colon s\widehat{g}(n)\geq 0 \text{ if }|n|_2\geq k\},
\end{align*}
(here, $|\cdot|_2$ denotes the Euclidean norm) together with the optimal constant
\begin{equation}\label{eq:AsTd}
\mathbb{P}_s(\mathbb{T}^d) 
:= \inf_{g \in \mathcal{P}_s(\mathbb{T}^d)\setminus\{{\bf 0}\}} \sqrt{r(g;\mathbb T^d)k_s(\widehat{g})}.
\end{equation}
A straightforward application of  \cite[Theorem 1.2]{GOeSR19} reveals that
$$
\mathbb{P}_s(\mathbb{T}^d)  \geq (1+o(1))\sqrt{\frac{d}{2\pi e}}.
$$
Identity $\mathbb{A}_+(1) = \mathbb{A}_+^{\mathcal B}(1)$ from Theorem \ref{thm:plus-equiv} leads to the following connection between the continuous and discrete versions of the one-dimensional $+1$ sign uncertainty principle.

\begin{prop}\label{thm:cont-disc} 
Given $s\in\{+,-\}$ and $d\geq 1$, it holds that
\[\mathbb{P}_s(\mathbb{T}^d)  \le \mathbb{A}^{\mathcal B}_s(d).\]
In particular,
\[\mathbb{P}_+(\mathbb{T}^1)  \le \mathbb{A}_+(1).\]
\end{prop}

\begin{proof}
Let $f\in\A_s(d)$ be nonzero and bandlimited. Assume that supp$(\widehat{f}) \subseteq [-a,a]^d$, for some $a>0$. Let $f_\lambda(x):=f(\lambda x)$, in which case $\widehat{f}_\lambda(\xi)=\lambda^{-d}\widehat f(\xi/\lambda)$ and supp$(\widehat{f}_\lambda)\subseteq [-a\lambda,a\lambda]^d$. 
Note that, as long as $2a\lambda<1$, $\widehat{f}_\lambda$ can be seen as a function on $\mathbb T^d\simeq[-\tfrac12,\tfrac12]^d$. Set $\lambda=\frac{r(f)}n$, where $n\in \N$ is sufficiently large  so that $2a\lambda<1$, and define $g(\xi):=s\widehat{f_\lambda}(\xi)$. 
By the higher dimensional version of a result of Plancherel \& P\'olya \cite{PP},  
\begin{equation}\label{PP}
 \sum_{k \in \mathbb{Z}^d} |\widehat{g}(k)| = \sum_{k\in\mathbb{Z}^d} |f(\lambda k)| < \infty,
 \end{equation}
whence $g\in\mathcal{P}_s(\mathbb{T}^d)\setminus\{{\bf 0}\}$. 
Moreover, $r(g;\mathbb T^d) = r(\widehat{f}_\lambda)=\frac{r(f)}n r(\widehat{f})$, and $k_s(\widehat{g})\leq n$. Therefore
\[\mathbb{P}_s(\mathbb{T}^d)^2 \le {r(g;\mathbb T^d)k_s(\widehat{g})} \le {r(f)r(\widehat{f})}.\]
Since  $f\in\A_s(d)$ was an arbitrary nonzero bandlimited function, this finishes the proof of the proposition.
\end{proof}

\section*{Acknowledgements} 
F.G.\@ acknowledges support from the Deutsche Forschungsgemeinschaft through the Collaborative Research Center 1060. 
D.O.S.\@ is supported by the EPSRC New Investigator Award ``Sharp Fourier Restriction Theory'', grant no.\@ EP/T001364/1.
J.P.G.R.\@ acknowledges financial support from the Deutscher Akademischer Austauschdienst.
The authors are indebted to the anonymous referees for careful reading and valuable suggestions.

\end{document}